\documentclass[10pt]{amsart}
\usepackage{a4wide}
\usepackage{amsfonts}
\usepackage{psfrag}

\usepackage{amsmath}
\usepackage{amssymb}
\usepackage{amsthm}
\usepackage{graphicx}
\newtheorem {definition}{Definition}

\newtheorem {lem}{Lemma}
\newtheorem {thm}{Theorem}
\newtheorem {remark}{Remark}

\newtheorem {example}{Example}

\newcommand{\E}{\mathbb{E}}

\newcommand{\R}{\mathbb{R}}
\newcommand{\T}{\mathbb{T}}
\newcommand{\1}{{\mathbf 1}}
\renewcommand{\P}{\mathbb{P}}
\newcommand{\CA}{\mathcal{A}}

\numberwithin{equation}{section}
\numberwithin{thm}{section}
\numberwithin{remark}{section}
\numberwithin{lem}{section}
\numberwithin{conj}{section}
\numberwithin{prop}{section}
\numberwithin{definition}{section}
\numberwithin{example}{section}

\newcommand{\Z}{\mathbb{Z}}
\newcommand{\ep}{\epsilon}

\begin{document}
\title[Longtime Behavior for Mutually Catalytic Branching]{Longtime Behavior for Mutually Catalytic Branching with Negative Correlations}
\author{Leif D\"oring}

\address{LPMA, Universit\'e Paris VI, Tours 16/26, 4 Place Jussieu, 75005 Paris, France}
\email{leif.doering@googlemail.com}
\thanks{LD was supported by the Fondation Science Mat\'ematiques de Paris}

\author{Leonid Mytnik}
\address{Faculty of Industrial Engineering and Management Technion Israel Institute of Technology, Haifa 32000, Israel}
\email{leonid@ie.technion.ac.il}
\thanks{LM is partly supported by the Israel Science Foundation and B. and G. Greenberg Research Fund (Ottawa)}

\maketitle

	\begin{abstract}
		In several examples, dualities for interacting diffusion and particle systems permit  the study of the longtime behavior of solutions. A particularly difficult model in which many techniques collapse is a two-type model 			with mutually catalytic interaction introduced by Dawson/Perkins for which they proved under some assumptions a dichotomy between extinction and coexistence directly from the defining equations.\\
		In the present article we show how to prove a precise dichotomy for a related model with negatively correlated noises. The proof uses moment bounds on exit-times of correlated Brownian motions from the first quadrant and explicit second moment calculations. Since the uniform integrability bound is independent of the branching rate our proof can be extended to infinite branching rate processes.
	\end{abstract}

\section{Introduction}				
	A classical task in the theory of interacting particle systems and interacting diffusion models is a precise understanding of the longtime behavior of the system. For many models dichotomies between extinction and 	survival have been revealed. In most cases proofs are based on clever duality constructions or explicit representations of the particular process. A good example is the {\it voter model} for which a graphical construction can 	be applied that 	reduces extinction problems to hitting problems of random walks (see for example Chapter V of \cite{L} for many beautiful results).\\
	In this article we aim at giving simple proofs for the dichotomy in a related class of mutually interacting diffusion processes with negatively correlated driving noises. Interestingly, the particular structure of the processes permits a second 	moment calculation that leads to the precise characterization of survival/extinction via recurrence/transience of the underlying migration mechanism. The approach is more direct than the previously used extension of arguments due 	to \cite{DP98} for which regularity assumption on the underlying migration mechanism were imposed.
	
\subsection{Finite Rate Symbiotic Branching Processes}
\label{sec:1.1}

	In 2004, Etheridge and Fleischmann \cite{EF04} introduced a stochastic spatial model of two interacting populations known as the (finite rate) {\it symbiotic branching model}, parametrized by a parameter $\varrho \in [-1,1]$ 			governing the correlations between the driving noises and a branching parameter $\gamma>0$ amplifying the strength of the noises. For a discrete spatial version of their stochastic heat equations, we consider the system of interacting diffusions on a 		countable set $S$ with values in $\R_{\geq 0}$, defined by the coupled stochastic differential equations
	\begin{equation*}
		\mathrm{SBM_\gamma}=\mathrm{SBM}(\varrho,\gamma)_{u_0, v_0} :\quad
	         \begin{cases}
			d u_t(i)=\mathcal A u_t(i) \, dt + \sqrt{\gamma u_t(i)v_t(i)} \, dB^1_t(i),\\[0.3cm]
			d v_t(i)=\mathcal A v_t(i) \, dt + \sqrt{\gamma u_t(i)v_t(i)}\,dB^2_t(i),\\[0.3cm]
		        	u_0(i) \ge 0, \quad i \in S,\\
        			v_0(i) \ge 0, \quad i \in S,
		\end{cases}
	\end{equation*}
	where $\big\{B^1(i),B^2(i)\big\}_{i \in S}$ is a $(\R^2)^{S}$-valued Gaussian process with covariance structure
	\begin{equation*}
		\E\big[B_{t}^n(i)B_{t}^m(j)\big] \, = \, 
		\begin{cases}
			\varrho t    &: i=j\text{ and } n \neq m,\\
			t            &: i=j\text{ and } n=m, \\
			0            &: \mbox{otherwise.}
		\end{cases}
	\end{equation*}
	The migration operator $\mathcal A$ is defined as
	\begin{align*}
		\mathcal A w(i)=\sum_{j\in S}a(i,j) w(j),
	\end{align*}
	where $\big(a(i,j)\big)_{i,j\in S}$ will always be assumed to be the $Q$-matrix of a symmetric continuous-time $S$-valued Markov 	process with bounded jump-rate, i.e.
	\begin{align*}
		\sup_{k\in S}|a(k,k)|<\infty.	
	\end{align*}
	Some care is needed to define properly the state-space of solutions. Here, we consider solutions in the space $L^\beta\subset(\R^2_+)^S$, usually referred to as Liggett-Spitzer space in the theory of interacting particle systems. Fixing a test-sequence $\beta\in (0,\infty)^S$  such that
	\begin{align*}
		\sum_{i\in S}\beta(i)<\infty\qquad\text{and}\qquad \sum_{i\in S}\beta(i)|a(i,k)|<M\beta(k)
	\end{align*}
	for all $k\in S$,  the state-space of solutions becomes 
	\begin{align*}
		L^\beta:=\big\{(u,v): S\to \R^+\times \R^+, \langle u,\beta\rangle,\langle v,\beta\rangle<\infty\big\},
	\end{align*}
	where $\langle f,g\rangle=\sum_{k\in S}f(k)g(k)$. The existence of the test-sequence $\beta$ is ensured by Lemma IX.1.6 of \cite{L}. For $u\in\R^S$, let
\begin{equation*}
\left\Vert u\right\Vert_{\beta,1} =\sum_{k\in S} |u(k)|\beta(k).
\end{equation*}
Furthermore, for $w=(u,v)\in L^{\beta}$, let $\Vert w\Vert_{\beta}=\Vert u\Vert_{\beta,1}+\Vert v\Vert_{\beta,1}$.
Note that $\Vert\cdot\Vert_\beta$ defines a topology on $L^{\beta}$. We will henceforth assume that $L^{\beta}$  is equipped with this topology.
 
	\smallskip
	
	We shall say that a pair $(u^\gamma_t,v^\gamma_t)$ of continuous $L^{\beta}$-valued adapted processes is a solution of $\textrm{SBM}_\gamma$ on the stochastic basis $(\Omega, \mathcal F, \mathcal F_t, \P)$ if there is 	a family $\big\{B^1(i),B^2(i)\big\}_{i \in S}$ of $\mathcal F_t$-Brownian motions with the aforementioned correlation structure and the stochastic equation is fulfilled. Hence, the solutions are defined  in the weak sense.
They are usually  first built on finite subsets of $S$ via standard SDE theory and then transferred to $S$ via a weak limiting procedure.
	To avoid long repetitions we refer the reader to Section 3 of \cite{BDE11} for a summary of existence of weak solutions, uniqueness and duality results for symbiotic branching processes in the case of  $ \mathcal A=\Delta $ being the discrete Laplacian; note that these results can be easily translated for the  general case of $\mathcal A$ treated here. We also refer the reader to~\cite{DM11} for a general review of the subject.  

	\medskip
	
	Let us be more specific about the migration operator $\mathcal A$. There is a number of typical choices for $\mathcal A$ which the reader might keep in mind. First of all it is the case of the discrete Laplacian
	\begin{align*}
		\mathcal A w(i)=\Delta w(i)=\sum_{|k-i|=1}\frac{1}{2d}(w(k)-w(i)), \;i\in \Z^d,
	\end{align*}
        which describes nearest neighbor interaction on $S=\Z^d$. There are also the cases of  
	complete graph interaction on a finite set $S$ corresponding to $a(i,j)=|S|^{-1}$, $i\neq j$, and the trivial migration $\mathcal A w=0$ on a single point set $S=\{s\}$ leading to the non-spatial symbiotic branching SDE
	\begin{eqnarray}\label{eqn:s}
		\begin{cases}
			du_t=\sqrt{\gamma u_tv_t}\,dB^1_t,\\
			dv_t=\sqrt{\gamma u_tv_t}\,dB^2_t,
		\end{cases}
	\end{eqnarray}
	driven by correlated Brownian motions.
	
	\medskip
	So far we did not motivate the reason to study this particular set of stochastic differential equations. Interestingly, symbiotic branching models relate spatial models from different branches of probability theory of the 	type
	\begin{eqnarray}\label{int}
		d w_t(i)=\mathcal A w_t(i)\,dt + \sqrt{\gamma f(w_t(i))}\,dB_t(i),\quad i\in S,
	\end{eqnarray}
	that are usually referred to as interacting diffusions models. Here are some particular examples:
	\begin{example}\label{ex1}
		The {\it stepping stone model} from mathematical genetics: $f(x)= x(1-x)$ (see for instance \cite{SS80}).
	\end{example}

	\begin{example}\label{ex2}
		The {\it parabolic Anderson model (with Brownian potential)} from mathematical physics: $f(x)= x^2$ (see for instance \cite{S92}).
	\end{example}

	\begin{example} \label{ex3}
		The {\it super-random walk} from the theory of branching processes: $f(x)= x$ (see for instance Section 2.2.4 of \cite{E00} for the continuum analogue).
	\end{example}

	For the super-random walk, $\gamma$ is the branching rate which in this case is time-space independent. In \cite{DP98}, a two-type model based on two super-random walks with time-space dependent branching was 			introduced. The branching rate for one species at a given site is proportional to the size of the other species at the same site. More precisely, the authors considered for $i\in S$
	\begin{align}
\label{mutcat}
\begin{cases}
		d u_t(i)=\mathcal A u_t(i) \, dt + \sqrt{\gamma u_t(i)v_t(i)} \, dB^1_t(i),\\
		d v_t(i)=\mathcal A v_t(i) \, dt + \sqrt{\gamma u_t(i)v_t(i)}\,dB^2_t(i),
\end{cases}
\end{align} 
	where now $\big\{B^1(i),B^2(i)\big\}_{i \in S}$ is a family of independent standard Brownian motions. Solutions are called {\it mutually catalytic branching} processes. The interest in mutually catalytic branching processes 	originates from the fact that it (resp. it's continuum analogue) constitutes a version of two interacting super-processes with random branching environment. Many of the classical tools from the study of super-processes fail since the branching 		property breaks down. Nonetheless, a detailed study is possible due to the symmetric nature of the equations. During the last decade, properties of this model were well studied (see for instance \cite{CK00} and \cite{CDG04}).\\
	Let us now see how the above examples relate to the symbiotic branching model $\mathrm{SBM_\gamma}$. For correlation $\varrho=0$, solutions of the symbiotic branching model are obviously solutions of the mutually catalytic 		branching model. The case $\varrho=-1$ with the additional assumption $u_0+v_0\equiv 1$ corresponds 	to the stepping stone model. To see this, observe that in this perfectly negatively correlated case of $B^1(i)=-B^2(i)$, the sum $u+v$ solves a discrete heat equation, and with the further assumption $u_0+v_0\equiv 	1$, $u+v$ stays constant for all time. Hence, for all $t\geq 0$, $u_t\equiv 1-v_t$ which shows that $u_t$ is a solution 	of the stepping stone model with initial condition $u_0$ and $v_t$ is a solution with initial condition $v_0$. Finally, suppose $w$ is a solution of the parabolic Anderson model, then, for $\varrho=1$, the pair 
	$(u,v):=(w,w)$ is a solution of the symbiotic branching model with initial conditions $u_0=v_0=w_0$.
	\smallskip

	This unifying property motivated the study in \cite{BDE11} of the influence of varying $\varrho$ on the longtime behavior for $\gamma$ being a fixed constant. Since the stepping stone model and the parabolic Anderson model have very different longtime properties it could be expected 	to recover some of those properties in disjoint regions of the parameter range. Restricting to $\mathcal A=\Delta$ on $\Z^d$ and $d=1,2$ the longtime behavior of laws and moments has been analyzed. An important observation of \cite{BDE11}, which holds equally for quite general $\mathcal A$ generating a recurrent Markov process, is the following moment transition for solutions $\mathrm{SBM_\gamma}$ started in 	homogeneous 	initial conditions $u_0=v_0\equiv 1$:
	\begin{align}\label{1}
	\quad \varrho<0\quad \Longleftrightarrow	\quad\text{There is }\epsilon>0\text{ such that }\E\big[u^\gamma_t(k)^{2+\epsilon}\big] \text{ is bounded in }t\text{ for all }k\in S.
	\end{align}
	A property of this kind could of course be expected: For $\varrho=-1$ solutions correspond to the stepping stone model which is bounded by $1$ and hence has bounded moments of all order. The other extremal case 
	$\varrho=1$ corresponds to the parabolic Anderson model that has exponentially increasing second moments (see for instance Theorem 1.6 of \cite{GdH07}).

\subsection{Infinite Rate Symbiotic Branching Processes}
	Looking more closely at the proofs of \cite{BDE11}, one observes that (\ref{1}) can be strengthened as follows:
	\begin{align}\label{2}
		\varrho<0\quad\Longrightarrow\quad	\text{There is }\epsilon>0\text{ such that }\E\big[u^\gamma_t(k)^{2+\epsilon}\big] \text{ is bounded in}\;  t\; \textbf{and} \;\gamma\text{ for all }k\in S,
	\end{align}
and this holds for any $\mathcal A$ satisfying the assumptions mentioned in the introduction (recall that $(u^{\gamma}, v^{\gamma})$ is a solution to $\mathrm{SBM}_{\gamma}$).
	This observation shall be combined in the following with a recent development for mutually catalytic branching processes which we now briefly outline.\\
	In \cite{KM10b}, for $\varrho=0$, existence of the {\it infinite rate mutually catalytic branching  processes}, appearing as limits in
	\begin{align}\label{eqn:23}
		\big(u^\gamma, v^\gamma\big)\stackrel{\gamma\to\infty}{\Longrightarrow} (u^\infty,v^\infty)
	\end{align}
	has been shown. This process and its properties have been further studied in~\cite{KM10a} and \cite{KO10}. In~\cite{DM11} the {\it infinite rate symbiotic branching processes} have been constructed for the 
 whole range of parameters $\varrho\in (-1,1)$. 
Below we state the result from~\cite{DM11} that introduces the  infinitete rate symbiotic branching process via the limiting procedure~(\ref{eqn:23}). Before doing this, let us shortly discuss why 
this procedure can lead to an exciting process.\\

	To understand the effect of sending $\gamma$ to infinity one might take a closer look at the non-spatial symbiotic branching SDE (\ref{eqn:s}). Due to the symmetric structure and the Dubins-Schwartz theorem, solutions $(u^\gamma,v^\gamma)$ can be regarded as time-changed correlated Brownian motions with time-change $\gamma \int_0^tu^\gamma_sv^\gamma_s\,ds$. The boundary of the first quadrant $$E=[0,\infty)\times[0,\infty) \backslash (0,\infty)\times (0,\infty)$$ is absorbing as the noise is multiplicative
 with diffusion coefficients proportional to the product of both coordinates.
  Hence, the pair of time-changed Brownian motions stops once hitting $E$. Increasing $\gamma$ only has the effect that the process follows the Brownian paths with higher speed so 	that $\gamma=\infty$ corresponds to immediately picking a point in $E$ according to the exit-point measure of the two-dimensional
  Brownian motion at the boundary of the first quadrant and the process stays at this point forever.\\
	Incorporating space, a second effect appears: both types are distributed in space according to the heat flow corresponding to $\mathcal A$. This smoothing effect immediately tries to lift a zero coordinate if it was pushed by the 	exit-measure to zero. 
	Interestingly, none of these two effects dominates and a non-trivial limiting process 
 is obtained via the limiting procedure (\ref{eqn:23}). 	

Before stating the existence theorem we need to refine the state-space:
$$	L^{\beta,E}:=L^\beta\cap E^S.$$
The space $L^{\beta,E}$  consists of sequences of pairs of points in $E$ (i.e. one coordinate is zero, one coordinate is non-negative) with restricted growth condition. The space $L^{\beta, E}$ is equipped  with the topology induced from the topology on  $L^{\beta}$ introduced in Section~\ref{sec:1.1}.   From the motivation given above for the infinite branching rate limiting behavior of the one-dimensional version (\ref{eqn:s}), the occurrence of the state-space $L^{\beta, E}$ is not surprising: the infinite rate symbiotic branching process, abbreviated as $\mathrm{SBM}_\infty$, takes values in $E$ at each fixed site $k\in S$.
Since in this paper we will be dealing with finite total mass processes, let us also  define
\begin{align*}
		L^1&:=\big\{(u,v): S\to \R^+\times \R^+, \langle u, 1\rangle,\langle v,1\rangle<\infty\big\},\\
	   	L^{1,E}&:=L^1\cap E^S.
\end{align*}
Additionally, we will denote by $D_{L^{\beta,E}}$ the space of RCLL functions on ${L^{\beta,E}}$.  
Before we state the  result from~\cite{DM11} on existence and uniqueness 
of the infinite rate symbiotic branching processes we need to introduce some additional notation. 
For $\varrho\in (-1,1)$, any $(x_1,x_2)\in L^{\beta}$ and any compactly supported  $(y_1,y_2)\in L^{1,E}$, set 
	\begin{align*}\begin{split}
		\langle\langle x_1,x_2,y_1,y_2\rangle\rangle_\varrho
		&= \sum_{k\in \Z^d}\Bigg[-\sqrt{1-\varrho}\big( x_1(k)+x_2(k)\big)\big(y_1(k)+y_2(k)\big)\\
		&\quad+i\sqrt{1+\varrho}\big( x_1(k)-x_2(k)\big)\big(y_1(k)-y_2(k)\big)\Bigg],\end{split}
	\end{align*}
and define $F(x_1,x_2,y_1,y_2)\equiv\exp(\langle\langle x_1,x_2,y_1,y_2\rangle\rangle_\varrho)$. Then we have the following result.

	
\begin{thm}
\label{thm:1.1}
		Suppose $\varrho\in(-1,1)$ and $\{(u^{\gamma}, v^{\gamma})\}_{\gamma>0}$ is a family of solutions to 
 ${\rm SBM}_{\gamma}$ with initial conditions $(u^{\gamma}_0,v^{\gamma}_0)=(u_0,v_0)\in L^{\beta,E}$ that do not depend on 
 $\gamma$. 
  For any sequence $\gamma_n$ tending to infinity, we have the convergence in law
	\begin{align*}
		(u^{\gamma_n},v^{\gamma_n})\Longrightarrow (u^{\infty},v^{\infty}),\quad n\to\infty,
	\end{align*}
	in $D([0,\infty),L^\beta)$ equipped with the Meyer-Zheng ``pseudo-path" topology.
The limiting process $
(u^{\infty},v^{\infty})$ is almost surely in  $D_{L^{\beta,E}}$
and it is the unique solution in $D_{L^{\beta,E}}$ to the following martingale problem:
		\begin{eqnarray*}
&&\left\{	\begin{array}{l}		
{\rm For\; any\; compactly\; supported \;} (y_1,y_2)\in L^{1,E},
\\
\\
F(u^{\infty}_t,v^{\infty}_t, y_1,y_2)-F(u_0,v_0, y_1,y_2)
-\int_0^t\langle \langle \mathcal A u^{\infty}_s,\mathcal A v^{\infty}_s,y_1,y_2\rangle\rangle_\varrho F(u^{\infty}_s,v^{\infty}_s,y_1,y_2)\,ds \\
\\
 \mbox{is a martingale null at zero.} 
\end{array}
\right.	
	\end{eqnarray*}
 \end{thm}
 Since the theorem is a direct combination of Theorems~3.4 and 3.6 from~\cite{DM11} its proof is omitted and we  will make just a few comments.
 The convergence of $(u^{\gamma},v^{\gamma})$ to $(u^{\infty},v^{\infty})$  is not stated in the Skorohod topology on  $D_{L^{\beta}}$ but instead in the Meyer-Zheng pseudo-path topology. In fact, the weak convergence is impossible in the Skorohod topology since $(u^\gamma,v^\gamma)$ are continuous processes whereas the limiting process $(u^\infty,v^\infty)$ only has RCLL paths (non-continuity does not directly become apparent but follows from a jump-type diffusion characterization) and in the Skorohod topology the subspace of continuous functions is closed.  For the Meyer-Zheng pseudo-path topology this, in fact, is possible (for more details see~\cite{MZ} and for other interesting developments see also~\cite{J97}).  
\smallskip

For the next theorem we assume that $\varrho<0$ and, in this case, some useful properties of  $(u^{\infty}, v^{\infty})$ with initial conditions in $L^{1,E}$ are established. 	
\begin{thm}\label{thm:existence}
		Suppose $\varrho<0$ and let 
$(u^{\infty}, v^{\infty})$ be the infinite rate process from Theorem~\ref{thm:1.1}
with initial conditions $(u^{\infty}_0,v^{\infty}_0)\in L^{1,E}$.
\begin{itemize}
 	\item[{a)}] The total-mass processes $\langle u^{\infty}_t, 1\rangle$ and $\langle v^{\infty}_t,1\rangle$ are well-defined square integrable martingales for $t\geq 0$.
	\item[{b)}]  There is an $\epsilon>0$ such that 
	$$\sup_{t>0} \E\left[ \langle u^{\infty}_t, 1\rangle^{2+\epsilon}\right]<\infty\quad \text{ and }\quad 
	 \sup_{t>0} \E\left[ \langle v^{\infty}_t, 1\rangle^{2+\epsilon}\right]<\infty.$$
	\item[{c)}]  If $\xi^1_t, \xi^2_t$ are independent continuous-time Markov processes with generator $\mathcal A$, then
	$$\E\left[ u^{\infty}_t(a) v^{\infty}_t(b)\right]= \E^{a,b}\big[ 
 u_0(\xi^1_t) v_0(\xi^2_t)  \mathbf{1}_{\{\xi^1_s\neq \xi^2_s,\forall s\leq t\}} \big]$$
 for any $t\geq 0$ and $a,b\in S$.
\end{itemize}
	\end{thm}
The second moment expression in part~{c)} of the theorem emphasizes the fact that $(u^\infty_t(a),v^\infty_t(a))\in E$ for any $a\in S,\; t\geq 0$. This becomes clear since for $a=b$, $\xi^1_0=\xi^2_0$ which gives $\E\left[ u^{\infty}_t(a) v^{\infty}_t(a)\right]=0$ and this, in turn, implies that  $ u^{\infty}_t(a) v^{\infty}_t(a)=0$ almost surely, for any $t\geq 0$. \\

\section{The Longtime Behavior for Negative Correlations}
	We now turn to the results of this article that address the longtime behavior of
 finite and infinite rate symbiotic branching processes. The question we address is classical for many particle systems on $[0,\infty)^S$: suppose a system has initial 			condition $w_0$ that is either infinite (e.g. $w_0\equiv 1$) or summable (e.g. $w_0=\mathbf 1_{\{k\}}$ for some $k\in S$), what can be said about limits of $w_t$ as $t$ tends to infinity? In many situations it turns out that 				equivalence holds between almost sure extinction for the (finite) total mass process $\langle w_t,1\rangle$ when $w_0=\mathbf 1_{k}$ and weak convergence of $w_t$ from constant initial states to the absorbing states of the system. Using different duality techniques this can be shown for instance for $\textrm{SBM}_\gamma$, the stepping stone model, the parabolic Anderson model and for the voter process. It is not known yet whether this property holds for $\textrm{SBM}_\infty$.
	\smallskip
	
	Extinction/survival dichotomies are typically of the following 
    type depending only on the recurrence/transience of the migration mechanism:
	\begin{align}\label{11}
	\begin{split}
		&\quad\lim_{t\to\infty}\langle 1,w_t\rangle\stackrel{a.s.}{=}0\,\,\,\,\text{for all summable initial conditions}\\
		&\Longleftrightarrow \quad\P^i(\xi_t=j\text{ for some }t \geq 0 )=1\quad \forall i,j\in S,
	\end{split}
	\end{align}
	where $\xi_t$ is a continuous-time Markov process with generator $\mathcal A$. 
		
\subsection{Some Known Results}
	Before stating our main theorem on the recurrence/transience dichotomy we recall some known results.
\smallskip

For $\gamma<\infty$, by writing the  stochastic equations for 
	$\textrm{SBM}_\gamma$ in mild form (see Lemma \ref{lem:3} below), it can be shown that if $(u_0^\gamma,v_0^\gamma)\in L^{1}$, then $\langle u^\gamma_t,1\rangle$ and $\langle v_t^\gamma,1\rangle$ are well-defined martingales which by non-negativity of solutions are non-negative.
 Also for $\gamma=\infty$, if
 $\varrho<0$ and additionally $(u^\infty_0,v^\infty_0)\in L^{1,E}$, the total-mass processes $\langle u^{\infty}_t,1\rangle$ and $\langle v_t^{\infty},1\rangle$ are non-negative martingales by Theorem~\ref{thm:existence}a). 
Hence, for $\gamma\in (0,\infty]$, by the martingale convergence theorem, $\langle u^{\gamma}_t,1\rangle$ and $\langle v_t^{\gamma},1\rangle$
converge almost surely, as $t\rightarrow \infty$,  and we denote by $\tilde u^\gamma_\infty, \tilde v^\gamma_\infty \in [0,\infty)$ their almost sure limits. 
Therefore,
	 \begin{align*}
	 	\lim_{t\to \infty}\langle u^\gamma_t,1\rangle \langle v_t^\gamma,1\rangle=\tilde u^\gamma_\infty
  \tilde v^\gamma_\infty\in [0,\infty)
	 \end{align*}
	 irrespectively of $\gamma$ being finite or infinite. In our two-type model this convergence is used to adapt the notion of
existence/survival from one-type models.
	\begin{definition}Let $\gamma\in(0,\infty]$, then we say coexistence of types is possible for $\mathrm{SBM}_\gamma$ if there are 
summable initial conditions $(u_0^\gamma, v_0^\gamma)$ such that $\tilde u^\gamma_\infty
\tilde v_\infty^\gamma >0$ with positive probability. 		Otherwise, we say that coexistence is impossible.
	\end{definition}
	We define the Green function of $\mathcal A$ by
	\begin{align*}
		g_t(j,k)=\int_0^t p_s(j,k)\,ds,
	\end{align*}
	where $p_s(j,k)=\P(\xi_s=k\,|\,\xi_0=j)$ and $\xi$ is a continuous-time Markov process with generator $\mathcal A$.
	The next theorem was the starting point for the longtime analysis of two-type models that we consider in this article.

	\begin{thm}[Theorem 1.2 of \cite{DP98}]\label{thm:dp}
		Suppose $\gamma<\infty$ and $\varrho=0$, then the following dichotomy holds:
		\begin{itemize}
			\item[a)] Transient case: If $\sup_{k\in S} g_\infty (k,k)<\infty$, then coexistence of types is possible.
			\item[b)] Recurrent case: Assuming additionally the uniformity condition
			\begin{align}\label{ass}
				g_T(j,j)\geq C \sup_{k\in S} g_T(k,k),\quad \forall j\in S,T\geq T_0(j),
			\end{align}
			on $\mathcal A$, a criterion for impossibility of coexistence of types is the following:
				\begin{align*}
					\P^j(\xi_t=i\text{ for some }t\geq 0)=1\quad\forall j,i\in S.
				\end{align*}					
		\end{itemize}
	\end{thm}
	The additional assumption~(\ref{ass}) 
   is fulfilled for instance for $\mathcal A=\Delta$ on $S=\Z^d$ so that coexistence in this case occurs if and only if $d\geq 3$.
	
	\begin{remark}
		The result of Theorem \ref{thm:dp} has already been partially generalized to $\varrho\neq 0$. For $\mathcal A=\Delta$ on $S=\Z^d$, the recurrent case (i.e. $d=1,2$) was dealt with in the proof of Proposition 2.1 of \cite{BDE11}. The proof tacitly uses the condition~(\ref{ass}). \\
		In the sequel we will present a different approach that only works for $\varrho<0$, and proves the full coexistence/non-coexistence dichotomy for general $\mathcal A$. The additional assumption (\ref{ass}) on $\mathcal A$ is not necessary in this case.
	\end{remark}

	
	Now let us switch to infinite rate processes. For $\varrho=0$ the longtime behavior was analyzed in \cite{KM10a} based on the approach of \cite{DP98} for $\gamma<\infty$. The proofs required more caution  
 than the proofs for $\gamma<\infty$ as all second moment arguments had to be 		avoided 
since  the infinite rate mutually catalytic branching process does not possess finite second moments. Such spatial systems are rare but interesting as their scaling properties change. The classical 			recurrence/transience dichotomy for finite variance models can break down in such a way that the criticality between survival and extinction is shifted to higher orders of the Green function (see for instance \cite{DF85}).
	To extend Theorem 1.2 of \cite{DP98}, the log-Green function is needed:
	\begin{align*}
		 g_{\infty,\log}(j,k)= \int_0^{\infty} p_s(j,k)(1+\vert \log{(p_s(j,k))}\vert)\,ds.
	\end{align*}
	Note that the log-Green function is infinite if the Green-function is infinite so that in the recurrent regime both are infinite.
	\begin{thm}[Theorem 1, Theorem 2 of \cite{KM10a}]\label{thm:km3}
		Suppose  that $\varrho=0$,  $(u^{\infty}_0,v^{\infty}_0)\in L^{1,E}$ and let $\CA$ be such that
		\begin{eqnarray*}
			 \sup_{k\in S}|a(k,k)\vert< 1\quad\text{and}\quad	\inf_{k\in S} |a(k,k)| >0. 
		\end{eqnarray*}

		\begin{itemize}
			\item[a)] If $g_{\infty,\log}(k,l)$ is "small  enough", then $\langle u^\infty_\infty,1\rangle \langle v_\infty^\infty,1\rangle>0$ with positive probability for localized initial conditions, i.e. there are $k,l\in S$ such that
			\begin{align*}
				\big(u^\infty_0,v^\infty_0\big)(i)=\begin{cases}
					(1,0)&: i=k,\\
					(0,1)&:i=l,\\
					(0,0)&:\textrm{otherwise}.
				\end{cases}
			\end{align*}
			In particular, for log-Green function "small enough", coexistence of types occurs.
		
			\item[b)] The "recurrent" regime holds as in part b) of Theorem \ref{thm:dp}.
		\end{itemize}
	\end{thm}
	
	Unfortunately, the description of the recurrence/transience dichotomy is even less precise than in Theorem \ref{thm:dp}. It remains open what happens in the case when $g_{\infty}(k,l)<\infty$ and $g_{\infty,\log}(k,l)=\infty$. 		Again, the simple random walk serves as an example for which the results can be clarified. Suppose $\mathcal A=\Delta$, then coexistence is impossible if $d=1,2$. For $d\geq 3$ one can use the local central limit theorem to show that $g_{\infty,\log}(k,j)\approx \|k-j\|^{2-d}$ as $\|k-j\|\to\infty$. 		Hence, the assumption of a) is fulfilled if initially the two populations are sufficiently far apart.
	

 \subsection{Main Result}
 	The main result of this article is a precise recurrence/transience dichotomy for $\varrho<0$ in both cases $\gamma<\infty$ and $\gamma=\infty$. The possibility of a dichotomy in terms of the  Green-function for $\gamma=\infty$ 	stems from the fact that for $\varrho<0$ the process has finite variance in contrast to infinite variance for $\varrho=0$.
	\begin{thm}\label{thm:3}
		Suppose $\varrho<0$ and $\gamma$ is either finite or infinite, then 
		\begin{align*}
			\text{coexistence of types is impossible}\qquad \Longleftrightarrow \qquad\P^j(\xi_t=i\text{ for some }t\geq 0)=1\quad\forall j,i\in S.
		\end{align*}
	\end{thm}
	  Interestingly, the negative correlations that seem to worsen the problem as symmetry gets lost, simplify the problem a lot so that	transparent proofs for this precise dichotomy are possible. 


	\subsection{An Open Problem: Longtime Behavior for Positive Correlations}\label{sec:pos}
		The reason we   skip the longtime behavior for positive correlations is simple: the longtime behavior even for the finite branching rate processes $\textrm{SBM}_\gamma$ is unknown for $\varrho>0$. Here we briefly describe  what might be 		expected for $\mathcal A=\Delta$   and $S=\Z^d$		but we do not have proofs.\\
		The particular case $\varrho=1$ corresponds to the classical parabolic Anderson problem (see Example 2). Having an explicit representation 
		of the solution as Feynman-Kac functional this problem	can be analyzed more easily. It is known that, started at localized initial conditions $w_0=\mathbf 1_{\{k\}}$, almost sure extinction
			$\lim_{t\to\infty}\langle w_t,1\rangle=0$ occurs if $d=1,2$ (the recurrent case). In \cite{S92}, see also \cite{GdH07}, it was proved that for $d\geq 3$ (the transient case) there are critical values 
		$\gamma^1(d)>\gamma^2(d)>0$ such that for 
		$\gamma>\gamma^1(d)$ extinction occurs almost surely and for $\gamma<\gamma^2(d)$ survival occurs with positive probability. This is one example where a "more noise kills" effect can be 			proved rigorously.
		
		\smallskip
		For $\textrm{SBM}_\gamma $ and $d\geq 3$, we conjecture the following: There should be a strictly decreasing critical curve $\gamma(\cdot,d):(0,1]\to \R^+$ coming down from infinity at zero and converging towards the 		critical threshold for the parabolic Anderson model at $1$ such that coexistence of types for $\textrm{SBM}_\gamma$ is impossible if $\gamma>\gamma(\varrho,d)$ and coexistence is possible if 
		$\gamma<\gamma(\varrho,d)$. 
		If the conjecture for $\gamma<\infty$ holds, it is furthermore natural to conjecture that for $\gamma=\infty$ coexistence of types is always impossible if $\varrho>0$.

	\section{Proofs}\label{sec:proofs}			
	
			\subsection{Some Properties of $\textrm{SBM}_{\gamma}$}\label{sec:finite}
			To prepare for the proofs of the longtime behavior we start with some lemmas for the finite branching rate processes $\textrm{SBM}_\gamma$. In order to avoid confusions with the notion $\langle f,g\rangle$ we denote the quadratic variation of square-integrable martingales by $[\cdot,\cdot]$. 
	
			\begin{lem}\label{lem:3} Suppose that $\varrho\in (-1,1)$,  $\gamma\in (0,\infty)$ and $(u^\gamma_0,v^\gamma_0)\in L^1$, then
			$\langle u^\gamma_t,1\rangle$ and $ \langle v^\gamma_t,1\rangle$ are non-negative martingales with cross-variations
			\begin{align*}
				\left[\langle u^\gamma_\cdot,1\rangle,\langle u^\gamma_\cdot,1\rangle\right]_t=\left[\langle v^\gamma_\cdot,1\rangle,\langle v^\gamma_\cdot,1\rangle\right]_t=\gamma \int_0^t \langle u^\gamma_s,v^\gamma_s				\rangle\,ds
			\end{align*}
			and			
			\begin{align*}
				\left[\langle u^\gamma_\cdot,1\rangle,\langle v^\gamma_\cdot,1\rangle\right]_t=\varrho\gamma \int_0^t \langle u^\gamma_s,v^\gamma_s\rangle\,ds.
			\end{align*}	
		
		\end{lem}
		\begin{proof}
			We only sketch a proof as it is rather straightforward. The proof is based on the stochastic variation of constant representation for stochastic heat equations. For
bounded  test-functions $\phi$ one can derive as in the proof of Theorem 2.2 of \cite{DP98} that 
			\begin{align*}
				\langle u^\gamma_t,\phi\rangle&=\langle u^\gamma_0,P_t\phi\rangle+\sum_{k\in S}\int_0^t P_{t-s}\phi(k)\sqrt{\gamma u^\gamma_s(k)v^\gamma_s(k)}\,dB^1_s(k),\\
				\langle v^\gamma_t,\phi\rangle&=\langle v^\gamma_0,P_t\phi\rangle+\sum_{k\in S}\int_0^tP_{t-s}\phi(k) \sqrt{\gamma u^\gamma_s(k)v^\gamma_s(k)}\,dB^2_s(k),
			\end{align*}
			where $P_tf(k)=\sum_{j\in S}p_t(k,j)f(j)$ and the Brownian motions are as in the definition of $\textrm{SBM}_\gamma$. The infinite sums can be shown to converge in $L^2(\P)$. Setting $\phi\equiv 1$ shows that $\langle u_t^\gamma,1\rangle$ and $\langle v_t^\gamma,1\rangle$ are infinite sums of 				martingales and for any finite subset $\Gamma$ of $S$ the correlation structure of the stochastic integrals can be easily calculated. As the martingale property is conserved under $L^2(\P)$-convergence and the cross-variations 			converge, we obtain that the total mass processes are square-integrable martingales with cross-variation
			\begin{align*}
				\big[\langle u^\gamma_\cdot,1\rangle,\langle v^\gamma_\cdot,1\rangle\big]_t&=\lim_{|\Gamma|\to\infty}\Big[\sum_{k\in \Gamma}\int_0^\cdot \sqrt{\gamma u^\gamma_s(k)v^\gamma_s(k)}\,dB^1_s(k),\sum_{j\in \Gamma}\int_0^\cdot 						\sqrt{\gamma u^\gamma_s(j)v^\gamma_s(j)}\,dB^2_s(j)\Big]_t\\
				&=\lim_{|\Gamma|\to\infty}\sum_{j,k\in \Gamma}\int_0^t \sqrt{\gamma u^\gamma_s(k)v^\gamma_s(k)}\sqrt{\gamma u^\gamma_s(j)v^\gamma_s(j)}\,d\big[B^1_\cdot(k),B^2_\cdot(j)\big]_s
			\end{align*}
			which by the presumed correlation structure equals
			\begin{align*}
				\varrho\lim_{|\Gamma|\to\infty}\sum_{k\in \Gamma}\int_0^t \gamma u^\gamma_s(k)v^\gamma_s(k)\,ds
				=\varrho\gamma \int_0^t \langle u^\gamma_s,v^\gamma_s\rangle\,ds.
			\end{align*}
			The derivation of the quadratic variations of $\langle u^\gamma_t,1\rangle$ and $\langle v^\gamma_t,1\rangle$ is similar, by dealing with the same sets of driving Brownian motions instead.
		\end{proof}
		
		The next lemma gives a refinement of Theorem 2.5 of \cite{BDE11} uniformly in $\gamma$.
		\begin{lem}\label{lem:moment}
			For any $\varrho<0$ there is an $\epsilon>0$ such that
			\begin{align*}
				\sup_{\gamma<\infty, T>0}\E\big[\sup_{t\leq T}\langle u^\gamma_t,1\rangle^{2+\epsilon}\big]<\infty,\\
				\sup_{\gamma<\infty, T>0}\E\big[\sup_{t\leq T}\langle v^\gamma_t,1\rangle^{2+\epsilon}\big]<\infty.
			\end{align*}
		\end{lem}
		\begin{proof}
			The cross-variation structure found for the martingales $\langle u^\gamma_t,1\rangle, \langle v^\gamma_t,1\rangle$ in the previous lemma allows us to obtain an upper bound for arbitrary (not only integer) moments 			by representing $(\langle u^\gamma_t,1\rangle,\langle v^\gamma_t,1\rangle)$ as a pair of time-changed correlated Brownian motions. A version of the Dubins-Schwartz theorem shows that 
			\begin{align}\label{hl}
				(W^1_t,W^2_t):=\big(\langle u^\gamma_{A_t},1\rangle, \langle v^\gamma_{A_t},1\rangle\big)
			\end{align}
			is a pair of correlated Brownian motions started in $(W^1_0, W^2_0)=(\langle u^\gamma_0,1\rangle, \langle v^\gamma_0,1\rangle)$ with covariance $\E[W^1_tW^2_t]=\varrho t$. Here, $A_t$ is the generalized inverse 			of 
			$[\langle v^\gamma_\cdot,1\rangle,\langle v^\gamma_\cdot,1\rangle]_t=[\langle u^\gamma_\cdot,1\rangle,\langle u^\gamma_\cdot,1\rangle]_t$. Furthermore, as the total-masses are non-negative, the time-change 				$A_t$ is bounded by $\tau$, the first hitting time of 
			$(W^1_t,W^2_t)$ at the boundary of the first quadrant (otherwise one of the total masses would become negative). Applying the Burkholder-Davis-Gundy inequality, this shows that
			\begin{align*}
				\sup_{\gamma<\infty,T>0}\E\big[\sup_{t\leq T}\langle u^\gamma_t,1\rangle^{p}\big]\leq cE\Big[\tau^{p/2}\,\Big|\,W^1_0=\langle u^\gamma_0,1\rangle,W_0^2=\langle v^\gamma_0,1\rangle\Big],
			\end{align*}
			which is finite if $p=2+\epsilon$ and $\epsilon$ is sufficiently small.	For independent Brownian motions the number of finite moments of the first exit time $\tau$ has been determined by Spitzer \cite{SP} (see also Burkholder~\cite{B77}); for the simple transformation to correlated Brownian motions we refer to 	Theorem 2.5 of \cite{BDE11}. 
		\end{proof}
		
		The final tool is a second moment formula for $\textrm{SBM}_\gamma$ for which we introduce the abbreviation $L_t=\int_0^t \mathbf 1_{\{\xi_s^1=\xi_s^2\}}\,ds$ for the collision time of two independent continuous-time Markov processes $\xi^1, \xi^2$ with generator $\mathcal A$.
		
		\begin{lem}\label{lem:mom}
		If $L_t$ denotes the collision time, then the second moment formula
			\begin{align}\label{eqn:mom}
				\E[u^\gamma_t(a)v^\gamma_t(b)]=\E^{a,b}\big[u^\gamma_0(\xi^1_t)v^\gamma_0(\xi^2_t)e^{\varrho\gamma L_t}\big],\quad a,b\in S,
			\end{align}holds.
		\end{lem}
		\begin{proof}
			There are several ways to see this expression. For instance, this follows as a particularly  simple application of the moment duality for $\textrm{SBM}_\gamma$ (derived in Proposition 9 of \cite{EF04}, see also the explanation in Lemma 3.3 of \cite{BDE11}).
		\end{proof}

		Formula (\ref{eqn:mom}) shows the significant difference of negative, null or positive correlations: only in the case of negative correlations one can hope to have finite second moment for the $\gamma=\infty$ 					limiting process.
		\smallskip
		
		With these preparations, we can proceed with the proof of Theorem~\ref{thm:existence}.
				
\subsection{Proof of Theorem~\ref{thm:existence}}
Let $\gamma_n\rightarrow\infty$, then we know from Theorem~\ref{thm:1.1} that
$$(u^{\gamma_n}, v^{\gamma_n})\Rightarrow (u^{\infty}, v^{\infty})$$
in $D_{L^{\beta}}$ in the Meyer-Zheng pseudo-path topology. 
\paragraph{\bf Proof of a), b):} 
First, we would like to show that the limiting total mass processes $\langle u^{\infty}_\cdot, 1\rangle$ and $\langle v^{\infty}_\cdot,1\rangle$ are  martingales. One should be a little bit careful, since convergence of  $(u^{\gamma_n}, v^{\gamma_n})$ is in $D_{L^{\beta}}$ and not  in $D_{L^{1}}$,
and hence not necessarily $(\langle u^{\gamma_n}_{\cdot}, 1\rangle, \langle v^{\gamma_n}_{\cdot}, 1\rangle)$ converges to $(\langle u^{\infty}_{\cdot}, 1\rangle, \langle v^{\infty}_{\cdot}, 1\rangle)$. However, without loss of generality we may and will assume that for our chosen subsequence $\{(u^{\gamma_n}, v^{\gamma_n})\}_{n\geq 1}$, at least,
\begin{eqnarray}
 \label{0609_1}
 \big(u^{\gamma_n}_{\cdot}, v^{\gamma_n}_{\cdot},\langle u^{\gamma_n}_{\cdot}, 1\rangle, \langle v^{\gamma_n}_{\cdot}, 1\rangle \big) \Rightarrow \big(u^{\infty}_{\cdot}, v^{\infty}_{\cdot},\bar u^{\infty}_{\cdot}, \bar v^{\infty}_{\cdot}\big)
\end{eqnarray}
in $D_{L^{\beta}}\times D_{\R}\times D_{\R}$ in the Meyer-Zheng pseudo-path topology. Let us show that indeed
\begin{eqnarray}
\label{0609_2}
\bar u^{\infty}_t = \langle u^{\infty}_t, 1\rangle,\;\;\; \bar v^{\infty}_t = \langle v^{\infty}_t, 1\rangle,\;\; 
\forall t\geq 0. 
\end{eqnarray}
By Theorem~5 of~\cite{MZ} the convergence in the Meyer-Zheng topology implies convergence (along a further subsequence which, again, we denote by $\gamma_n$ for convenience) of one-dimensional distributions  on a set of full 
Lebesgue measure, say $\T$. Fix any $t\in \T$. Since $u_0$ and $v_0$ are summable, for any $\ep>0$, one can fix a  sufficiently large compact $\Gamma\subset S$ such that 
\begin{eqnarray*}
\E\left[ \langle u^{\gamma_n}_{t}, \1_{\Gamma^c}\rangle + \langle v^{\gamma_n}_{t}, \1_{\Gamma^c}\rangle \right] = 
  \E\left[ \langle u^{\gamma_n}_{0}, P_t\1_{\Gamma^c}\rangle + \langle v^{\gamma_n}_{0}, P_t\1_{\Gamma^c}\rangle \right] 
\leq  \ep, \quad \forall \gamma_n>0. 
\end{eqnarray*}
For the equality we have used the fact $\E[\langle u_t^\gamma,\phi\rangle]=\langle u^\gamma_0,P_t\phi\rangle$ which follows from the stochastic variation of constant representation utilized in the proof of Lemma \ref{lem:3}.
This implies that, in fact, $ (u^{\gamma_n}_{t},  u^{\gamma_n}_{t})\Rightarrow (u^{\infty}_{t},  u^{\infty}_{t})$ 
in $M_F(S)\times M_F(S)$ - the product space of finite measures on $S$ equipped with the weak topology. This immediately implies
\begin{align}\label{dis}
	(\langle u^{\gamma_n}_{t}, 1\rangle, \langle v^{\gamma_n}_{t}, 1\rangle) \Rightarrow (\langle u^{\infty}_{t}, 1\rangle, \langle v^{\infty}_{t}, 1\rangle),\;\; \forall t\in \T,
\end{align}
and combined with (\ref{0609_1}) this gives
$$ (\bar u^{\infty}_{t}, \bar v^{\infty}_{t}) =  (\langle u^{\infty}_{t}, 1\rangle, \langle v^{\infty}_{t}, 1\rangle), \;\; \forall t\in \T.$$ 
Since both processes, on the right and on the left hand sides, are right-continuous we get that in fact the equality holds for all $t$, and hence (\ref{0609_2}) follows.

By the above, Theorem~11 of~\cite{MZ} and Lemma~\ref{lem:moment} we immediately get that the limiting 
total mass processes $\langle u^{\infty}_t, 1\rangle, \langle v^{\infty}_t, 1\rangle,\;t\geq 0,$ are martingales. We will get the square-integrability of these martingales (and thus finish the proof of {a)}) when we prove {b)}.  
For {b)} we use Fatou's lemma combined with (\ref{dis}) and Lemma~\ref{lem:moment} to obtain that
$$\sup_{t\in \T} \E\left[ \langle u^{\infty}_t, 1\rangle^{2+\epsilon}+ \langle v^{\infty}_t, 1\rangle^{2+\epsilon}\right]<\infty. $$
Recalling that $t\mapsto (\langle u^{\infty}_t, 1\rangle, \langle v^{\infty}_t, 1\rangle) $ 
is right-continuous, one can take the supremum over $t\in \R^+$, and {b)} follows.

\paragraph{\bf Proof of c):}
The representation is first deduced for $t\in \T$. By the choice of $\T$ and (\ref{0609_1}) we get 
\begin{eqnarray} 
u^{\gamma_n}_t(a) v^{\gamma_n}_t(b) \Rightarrow  u^{\infty}_t(a) v^{\infty}_t(b). 
\end{eqnarray}
To get convergence of the first moments of $\{u^{\gamma_n}_t(a) v^{\gamma_n}_t(b)\}_{n\geq 1}$ it is enough to check the uniform 
integrability of $\{u^{\gamma_n}_t(a) v^{\gamma_n}_t(b)\}_{n\geq 1}$. However this follows from 
Lemma \ref{lem:moment} and H\"older's inequality:
		\begin{align*}
			\sup_{n\geq 1}\E\left[\left(u^{\gamma_n}_t(a) v^{\gamma_n}_t(b)\right)^{\frac{2+\epsilon}{2}}\right]\leq 
  \sup_{n\geq 1}\sqrt{\E\big[\langle u^{\gamma_n}_t,1\rangle^{2+\epsilon}\big]\E\big[\langle v^{\gamma_n}_t,1\rangle^{2+\epsilon}\big]}<\infty,
		\end{align*}
		if $\epsilon$ is chosen sufficiently small.
The uniform integrability, Lemma~\ref{lem:mom} and the dominated convergence then imply  
		\begin{align*}
			\E[u^\infty_t(a)v^\infty_t(b)]&=\lim_{n\to\infty}\E[u^{\gamma_n}_t(a)v^{\gamma_n}_t(b)]\\
			&=\lim_{n\to\infty}\E^{a,b}\big[u^{\gamma_n}_0(\xi^1_t)v^{\gamma_n}_0(\xi^2_t)e^{\varrho\gamma_n L_t}\big]\\
			&=\E^{a,b}\big[u^\infty_0(\xi^1_t)v^{\infty}_0(\xi^2_t)\mathbf 1_{\{L_t=0\}}\big]+\E^{a,b}\big[u^{\infty}_0(\xi^1_t)v^{\infty}_0(\xi^2_t)e^{\lim_{n\to\infty}\gamma_n \varrho L_t}\mathbf 1_{\{L_t>0\}}\big]\\
			&=\E^{a,b}\big[u^\infty_0(\xi^1_t)v^\infty_0(\xi^2_t)\mathbf 1_{\{L_t=0\}}\big],
		\end{align*}
	where we also used that by definition $u^{\gamma_n}_0=u^\infty_0$ and $v^{\gamma_n}_0=v^\infty_0$.
	\subsection{Proof of Theorem \ref{thm:3}}
		With the preparations of Section \ref{sec:finite} we can now prove the extinction/coextinction dichotomy. 
	\subsubsection{Proof of Theorem \ref{thm:3}, $\gamma<\infty$}
	Recall that due to the martingale convergence theorem the product $M^\gamma_t:=\langle u^\gamma_t,1\rangle\langle v_t^\gamma,1\rangle$ of the two non-negative martingales 
		 converges almost surely to $\tilde u^\gamma_\infty
  \tilde v^\gamma_\infty$, and our task is to determine when the limit equals zero almost surely. As the limit is non-negative the most 		straightforward approach is to deduce a formula for $\E\big[M_\infty^\gamma\big]$ and to determine when this is strictly positive or null. For this to work, the assumption $\varrho<0$ is crucial.
		
Luckily, using the results from Section \ref{sec:finite}, the convergence of $\E[M^\gamma_t]$ to $\E[M^\gamma_\infty]$ comes almost for granted. The convergence of $M^\gamma_t$ holds almost surely, so it suffices to show uniform integrability in $t$. Choosing $\epsilon$ small enough, we obtain the uniform integrability from 
		H\"older's inequality and Lemma \ref{lem:moment}:
		\begin{align*}
			\E\Big[(M_t^\gamma)^{\frac{2+\epsilon}{2}}\Big]\leq \sqrt{\E[\langle u_t^\gamma,1\rangle^{2+\epsilon}]\E[\langle v^\gamma_t,1\rangle^{2+\epsilon}]}\leq C.
		\end{align*}			
		Combined with non-negativity of $\tilde u^\gamma_\infty
  \tilde v^\gamma_\infty$ this implies that
		\begin{align*}
			\tilde u^\gamma_\infty
  \tilde v^\gamma_\infty=0\quad a.s.\qquad\Longleftrightarrow\qquad \lim_{t\to\infty}\E[M_t^\gamma]=0.
		\end{align*}
		The first moment of $M_t^\gamma$ can be calculated from the moment-duality for $\textrm{SBM}_\gamma$ in such a way that the criterion of the dichotomy directly drops out; to finish the proof, we aim at using Lemma \ref{lem:mom} to show that 
		\begin{align*}
			 \lim_{t\to\infty}\E[M_t^\gamma]=0\text{ for all }
 (u^\gamma_0,v^\gamma_0)\in L^{1}\quad \Longleftrightarrow  \quad\P^{i,j}(\xi^1_t=\xi_t^2\text{ for some }t\geq 0)=1\quad\forall i,j\in S.
		\end{align*}
Taking into account Lemma \ref{lem:mom}, we get, for any $(u^\gamma_0,v^\gamma_0)\in L^1$ with $\langle u_0^\gamma,1\rangle\langle v^\gamma_0,1\rangle>0$,
		\begin{align*}
			\E[M_t^\gamma]&=\sum_{a,b\in S}\E[u^\gamma_t(a) v^\gamma_t(b)]\\
			&=\sum_{a,b\in S}\E^{a,b}\big[u_0^\gamma(\xi_t^1)v_0^\gamma(\xi_t^2)e^{\gamma\varrho L_t}\big]\\		
			&=\sum_{a,b\in S}\sum_{i,j\in S}u^{\gamma}_0(i)v^{\gamma}_0(j)\E^{a,b}\big[\mathbf 1_i(\xi_t^1)\mathbf 1_j(\xi_t^2)e^{\gamma\varrho L_t}\big]\\
			&=\sum_{i,j\in S}\sum_{a,b\in S}u^{\gamma}_0(i)v^{\gamma}_0(j)\E^{i,j}\big[\mathbf 1_a(\xi_t^1)\mathbf 1_b(\xi_t^2)e^{\gamma\varrho L_t}\big]\\				
			&=\sum_{i,j\in S}u^{\gamma}_0(i)v^{\gamma}_0(j)\E^{i,j}\big[e^{\gamma\varrho L_t}\big]. \;\; 
		\end{align*}
		For the fourth equation we used reversibility of the paths of $\xi^1,\xi^2$ since we assumed $\mathcal A$ to be symmetric with bounded jump rate. As $\varrho$ is negative, by the dominated convergence, and taking into account $\sum_{i,j}u^{\gamma}_0(i)v^{\gamma}_0(j)\in (0,\infty)$, we see that the right hand side converges to zero if and only if $\lim_{t\to\infty}\E^{i,j}\big[e^{\gamma\varrho L_t}\big]=0$	for all $i,j\in S$. The proof can now be finished via the dominated convergence again and the Markov property:
		\begin{align*}
			\lim_{t\to\infty}\E^{i,j}\big[e^{\gamma\varrho L_t}\big]=0\quad \forall i,j\in S\qquad &\Longleftrightarrow \qquad\P^{i,j}(L_t\to\infty)=1\quad \forall i,j\in S\\
			&\Longleftrightarrow \qquad  \P^{i,j}(\xi^1_t=\xi_t^2\text{ for some }t\geq 0)=1\quad \forall i,j\in S.
		\end{align*}			
		
	\subsubsection{Proof of Theorem \ref{thm:3}, $\gamma=\infty$}
		The proof of the dichotomy on the level of infinite branching rate can now be deduced similarly to the $\gamma<\infty$ case
		%
for which we define
$$ M_t^\infty:= \langle 	u^\infty_t, 1\rangle \langle v^\infty_t, 1\rangle$$
for $t\geq 0$.	
	\begin{lem}\label{lll}
		Suppose $\varrho<0$ and for the initial conditions $(u^\infty_0,v^\infty_0)\in L^{1,E},$ then
		\begin{align}\label{eqn:mominf}
			\E[M_t^\infty]=\sum_{i,j\in S, i\neq j}u^\infty_0(i)v^\infty_0(j)\P^{i,j}(\xi^1_s\neq \xi^2_s,\forall s\leq t)
		\end{align}
		for $t\geq 0$.
	\end{lem}
	\begin{proof}
		 We use Theorem~\ref{thm:existence}c), Fubini's theorem and the time-reversion trick used in the proof of Theorem \ref{thm:3} for $\gamma<\infty$:
		\begin{align*}
			\E[M_t^\infty]&=\sum_{a,b\in S}	\E[u^\infty_t(a)v^\infty_t(b)]\\
			&=\sum_{a,b\in S}\sum_{i,j\in S}\E^{a,b}\big[u^\infty_0(i) \mathbf 1_{i}(\xi^1_t)v^\infty_0(j) \mathbf 1_j(\xi^2_t)\mathbf 1_{L_t=0}\big]\\
			&=\sum_{i,j\in S}\sum_{a,b\in S}u^\infty_0(i)v^\infty_0(j)\E^{i,j}\big[\mathbf 1_{a}(\xi^1_t)\mathbf 1_b(\xi^2_t)\mathbf 1_{L_t=0}\big]\\
			&=\sum_{i,j\in S}u^\infty_0(i)v^\infty_0(j)\P^{i,j}\big[\xi^1_s\neq \xi^2_s,\forall s\leq t\big].
		\end{align*}
	By assumption $(u^\infty_0, v^\infty_0)\in E^S$ so that the terms with $i=j$ vanish as then $u^\infty_0(i)v^\infty_0(i)=0$.
	\end{proof}
	Now we have to verify necessary and sufficient conditions for the limit
	\begin{align*}
		\lim_{t\to\infty}\langle u^\infty_t,1\rangle\langle v^\infty_t,1\rangle=\tilde u^\infty_\infty
  \tilde v^\infty_\infty
	\end{align*}
	being equal to zero almost surely. This can be done similarly as in  the case $\gamma<\infty$.
	\begin{lem}\label{lem:h}
		Suppose $\varrho<0$ and for the initial conditions $(u^\infty_0,v^\infty_0)\in L^{1,E}$, then
		\begin{align*}
			\tilde u^\infty_\infty
  \tilde v^\infty_\infty=0\quad a.s.\qquad\Longleftrightarrow\qquad \lim_{t\to\infty}\E[M_t^\infty]=0.
		\end{align*}
	\end{lem}
	\begin{proof}
		Almost sure convergence of $M^\infty_t$ to $\tilde u^\infty_\infty
  \tilde v^\infty_\infty$ is due to the martingale convergence theorem, so that it suffices to show that $M^\infty_t$ is uniformly integrable in $t$. By H\"older's 			inequality we reduce the question to the total mass processes:
		\begin{align}\label{eqn:b}
			\E\left[(M^\infty_t)^{\frac{2+\epsilon}{2}}\right]\leq \sqrt{\E\big[\langle u^\infty_t,1\rangle^{2+\epsilon}\big]\E\big[\langle v^\infty_t,1\rangle^{2+\epsilon}\big]},
		\end{align}
and the result follows from Theorem~\ref{thm:existence}c) if $\epsilon$ is chosen small enough.
\end{proof}
	Now we can proceed with the proof of Theorem \ref{thm:3} for $\gamma=\infty$ exactly as for $\gamma<\infty$.
	
	\begin{proof}[Proof of Theorem \ref{thm:3}, $\gamma=\infty$]
		The theorem follows directly from Lemma \ref{lll} and Lemma \ref{lem:h}. Using the dominated convergence, justified by  $\sum_{i,j\in S}u_0^{\infty}(i)v_0^{\infty}(j)<\infty$, and monotonicity of measures we obtain
		\begin{align*}
			&\quad\lim_{t\to\infty}\sum_{i,j\in S, i\neq j}u_0^{\infty}(i)v_0^{\infty}(j)\P^{i,j}(\xi_s^i\neq \xi_s^2,\,\forall s\leq t)=0,\quad\forall (u_0^{\infty},v_0^{\infty})\in L^{1,E}\\
			&\Longleftrightarrow \qquad  \P^{i,j}(\xi^1_t=\xi_t^2\text{ for some }t\geq 0)=1\quad\forall j,i\in S, i\neq j.
		\end{align*}
		This finishes the proof of Theorem \ref{thm:3} also for infinite branching rate.
	\end{proof}

\section*{Acknowlegement}	
	The authors thank an anonymous referee for a very careful reading of the manuscript.

\end{document}